\documentclass[12pt]{amsart}

\usepackage[latin1]{inputenc}
\usepackage{amsmath,amsfonts,amssymb}

\usepackage{amsmath}
\usepackage{amssymb}
\usepackage{amsthm}
\usepackage{latexsym}
\usepackage{float}
\usepackage{pstricks}
\usepackage[draft]{graphicx}
\usepackage{tikz,pgf}
\def\opn#1#2{\def#1{\operatorname{#2}}} 
\opn\diam{diam}
\opn\last{last}
\opn\first{first}
\opn\pff{proof}
\opn\Pf{proof} \opn\GL{GL} \opn\SL{SL} \opn\mod{mod} \opn\ord{ord}

\newtheorem{thm}{Theorem}[section]
\newtheorem{cor}[thm]{Corollary}

\newtheorem{prop}[thm]{Proposition}

\theoremstyle{definition}
\newtheorem{defn}[thm]{Definition}
\theoremstyle{remark}
\newtheorem{rem}[thm]{Remark}

\numberwithin{equation}{section}

\let\epsilon\varepsilon
\let\phi=\varphi
\let\kappa=\varkappa
\def\qed{\ifhmode\textqed\fi
   \ifmmode\ifinner\quad\qedsymbol\else\dispqed\fi\fi}
\def\textqed{\unskip\nobreak\penalty50
    \hskip2em\hbox{}\nobreak\hfil\qedsymbol
    \parfillskip=0pt \finalhyphendemerits=0}
\def\dispqed{\rlap{\qquad\qedsymbol}}

\begin{document}

\title[Closed orders and closed graphs]
 {Closed orders and closed graphs $^1$}


\author[Marilena Crupi]{Marilena Crupi}
\address[Marilena Crupi]{Department of Mathematics and Computer Science,
Universiy of Messina\\ Viale Ferdinando Stagno d'Alcontres, 31\\ 98166 Messina, Italy.}
\email{mcrupi@unime.it}



\thanks{\textbf{$^1$ To appear in Analele Stiintifice ale Universitatii Ovidius Constanta}}

\maketitle
%

\begin{abstract}  The class of closed graphs by a linear ordering on their sets of vertices is investigated. 
A recent characterization of such a class of graphs is analyzed by using tools from the proper interval graph theory.
\end{abstract}

\section*{Introduction}
Let $G$ be a simple graph with finite vertex set $V(G)$ and  edge set $E(G)$. 
In the last years, several authors have focused their attention on the class of \emph{closed ideals} (see, for instance, \cite{CE, CR, CR2, HEH, HH} and the references therein). Let $S=K[x_1, \ldots, x_n,y_1,\ldots, y_n]$ be a polynomial ring with coefficients in a field $K$. The closed graphs were introduced in \cite{HH} in order to characterize those graphs, which, for suitable labeling of their edges, do have a quadratic Gr\"obner basis with respect to the lexicographic order induced by $x_1 > \cdots >x_n>y_1 > \cdots >y_n$. Such a  class of graphs is strictly related to the so-called \emph{binomial edge ideal} \cite{HH}. The binomial  edge ideal of a labeled graph $G$, denoted by  $J_G$, is the ideal of $S$ generated by the binomials $f_{ij} = x_iy_j - x_jy_i$ such that $i<j$ and $\{i,j\}$ is an edge of $G$.
In \cite{CR}, the authors have shown that the existence of a quadratic Gr\"obner basis for 
$J_G$ is not related to the lexicographic order on $S$. Indeed, one of the main results in the paper implies that the closed graphs are the only graphs for which $J_G$ has a quadratic Gr\"obner basis for some monomial order on $S$.  
Afterwards, the same  authors have proved that the class of closed graphs is isomorphic to the well known class of \emph{proper interval graphs} \cite{CR2}. 
Such a class of graphs has been strongly studied from both the theoretical and algorithmic point of view and many linear-time algorithms for proper interval graph recognition have been developed (see, for instance \cite{BL, FG, Ha, HMP, PD}  and the references therein).  Hence, the isomorphism in \cite{CR2} implies the existence of linear-time algorithms also for the closed graph recognition. 
In this note we study the behavior of closed graphs by using methods typical of the class of proper interval graphs.

The plan of the paper is the following. Section \ref{sec:pre} contains some preliminary notions that will be used in the paper. In Section \ref{sec:proper}, we discuss some results on proper interval graphs in order to compare such a class of graphs with the class of closed graphs. We introduce the \emph{closed orderings} (Definition \ref{def:closed}) and observe that closed orderings and \emph{proper interval orderings} (Definition \ref{def:proper}) coincide; as a consequence, we recover the isomorphism between the class of closed graphs and the class of proper interval graphs (Corollary \ref{cor:equiv}). In Section \ref{sec:inter}, we show how a recent characterization of closed graphs,  due to Cox and Erskine \cite{CE}, can be obtained via some properties of proper interval graphs (Theorem \ref{thm:narrow2}).

\section{Preliminaries} \label{sec:pre}
In this Section, we collect some notions that will be useful in the development of the paper.

Let $G$ be a simple, finite graph. Denote by $V(G)$ the set of vertices of $G$ and by $E(G)$ its edge set.
Let $v, w \in V(G)$. A \emph{path} $\pi$ of \emph{length} $n$ from $v$ to $w$ is a sequence of vertices $v=v_0, v_1, \ldots, v_n=w$ such that $\{v_i, v_{i+1}\}$ is an edge of the underlying graph.  A path $v_0, v_1, \ldots, v_n$ is \emph{closed} if $v_0=v_n$. A graph $G$ is \textit{connected} if for every pair of vertices $u$ and $v$ there is a path from $u$ to $v$.
The \emph{distance} $d(u,v)$ between vertices $u,v$ of a graph $G$ is the length of the shortest path connecting
them, and the \emph{diameter} $\diam(G)$ of $G$ is the maximum distance between two vertices
of $G$. A \emph{cycle} of \emph{length} $n$ is a closed path $v_0, v_1, \ldots, v_n$ in which $n\geq 3$. $G$  is \emph{chordal} or \emph{triangulated} if its cycles of four or more vertices has a \emph{chord}, which is an edge joining two non adjacent vertices of the cycle. 
Finally, $G$ is \emph{claw-free} (or \emph{net-free}, or \emph{tent-free}, respectively) if $G$ does not contain as induced subgraph the \emph{claw} (or the \emph{net}, or the \emph{tent}, respectively):
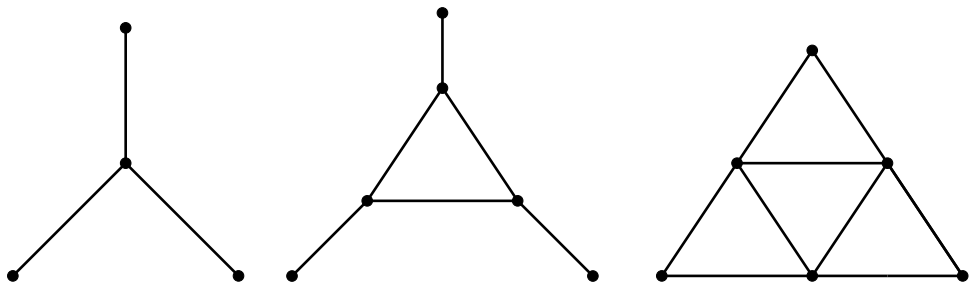
\begin{figure}[H]
\centering
\begin{tabular}{ccc}
\begin{tikzpicture}
\filldraw 
(0,0) circle (2pt) node[below left][black] {}
(1.5,1.5) circle (2pt) node[below right][black] {}
(1.5,3.3) circle (2pt) node[above right][black] {}
(3,0) circle (2pt) node[below right][black] {};
\draw [line width=1pt] (0,0) -- (1.5,1.5) node[midway,left] {};
\draw [line width=1pt] (1.5,1.5) -- (3,0) node[midway,right] {};
\draw [line width=1pt] (1.5,1.5) -- (1.5,3.3) node[midway,right] {};
\end{tikzpicture} &
\begin{tikzpicture}
\filldraw (-0.5,-0.5) circle (2pt) node[below left][black] {}
(-1.5,-1.5) circle (2pt) node[below right][black] {}
(2.5,-1.5) circle (2pt) node[below right][black] {}
(0.5,2) circle (2pt) node[below right][black] {}
(0.5,1) circle (2pt) node[above right][black] {}
(1.5,-0.5) circle (2pt) node[below right][black] {};
\draw [line width=1pt] (-0.5,-0.5) -- (0.5,1.0) node[midway,left] {};
\draw [line width=1pt] (0.5,1.0) -- (1.5,-0.5) node[midway,right] {};
\draw [line width=1pt] (-0.5,-0.5) -- (1.5,-0.5) node[midway,right] {};
\draw [line width=1pt] (-1.5,-1.5) -- (-0.5,-0.5) node[midway,left] {};
\draw [line width=1pt] (1.5,-0.5) -- (2.5,-1.5) node[midway,right] {};
\draw [line width=1pt] (0.5,1.0) -- (0.5,2) node[midway,right] {};
\end{tikzpicture} &
\begin{tikzpicture}
\filldraw (0,0) circle (2pt) node[below left][black] {}
(1,1.50) circle (2pt) node[above right][black] {}
(2,0) circle (2pt) node[below right][black] {}
(3,1.5) circle (2pt) node[below right][black] {}
(2,3)circle (2pt) node[below right][black] {}
(4,0)circle (2pt) node[below right][black] {};
\draw [line width=1pt] (0,0) -- (1,1.50) node[midway,left] {};
\draw [line width=1pt] (1,1.50) -- (2,0) node[midway,right] {};
\draw [line width=1pt] (0,0) -- (3,0) node[midway,right] {};
\draw [line width=1pt] (2,0) -- (3,1.50) node[midway,right] {};
\draw [line width=1pt] (1,1.50) -- (3,1.5) node[midway,right] {};
\draw [line width=1pt] (3,0) -- (4,0) node[midway,right] {};
\draw [line width=1pt] (3,1.50) -- (4,0) node[midway,right] {};
\draw [line width=1pt] (1,1.5) -- (2,3) node[midway,right] {};
\draw [line width=1pt] (2,3) -- (4,0) node[midway,right] {};
\end{tikzpicture}
\end{tabular}
\caption{\label{claw}\emph{The claw} (left), \emph{the net} (middle), \emph{the tent} (right).}
\end{figure}
\section{Closed orderings}\label{sec:proper}
In this Section, we analyze the isomorphism between the class of closed graphs and the class of proper interval graphs \cite{CR2} by vertex orderings. 

Set $[n] = \{1, \ldots,n \}$. A \emph{vertex ordering} (also called  a \emph{labeling}) of a $n$-vertex graph $G$ is a bjection $\sigma : V(G) \rightarrow [n]$. We write $u<_{\sigma}v$ to mean that $\sigma(u) < \sigma(v)$, for $u,v \in V(G)$.
If $\sigma$ is a vertex ordering of a $n$-vertex graph $G$, the vertices of $G$ can be ordered as $v_1, v_2, \ldots, v_n$ such that
\[v_1 <_{\sigma}v_2 <_{\sigma} \cdots <_{\sigma} v_n.\]
\begin{defn} \label{def:proper} Let $\sigma$ be a vertex ordering of a graph $G$. Ordering $\sigma$ is called a \emph{proper interval ordering} if for every triple $u, v,w$ of vertices of $G$ where $u <_{\sigma} v <_{\sigma} w$
and $\{u, w\}\in E(G)$, one has $\{u,v\}, \{v,w\} \in E(G)$.
\end{defn}
\begin{rem} \label{rem:closed} Let $\sigma$ be a vertex ordering of a graph $G$ and let
\[v_1 <_{\sigma}v_2 <_{\sigma} \cdots <_{\sigma} v_n\]
be the ordering of its vertices. $\sigma$ is a \emph{proper interval ordering} if for every triple $v_i, v_j, v_k$ of vertices of $G$, with $i < j < k$ and $\{v_i, v_k\}\in E(G)$, one has $\{v_i, v_j\}, \{v_j, v_k\} \in E(G)$.
\end{rem}

The vertex ordering above defined gives an important characterization of the so-called \emph{proper interval graphs}.

\begin{defn} A graph $G$
is an \emph{interval graph} if to each vertex $v \in V(G)$ it is possible to associate a closed interval $I_v = [\ell_v, r_v]$ of the real line such that two distinct vertices $u, v \in V(G)$ are adjacent if
and only if $I_u \cap I_v \neq \emptyset$.
\end{defn}

The family $\{I_v\}_{v\in V(G)}$ is \textit{an interval representation}
of $G$.

\begin{defn}
A graph $G$ is a \emph{proper interval graph} if there is an interval
representation of $G$ in which no interval properly contains other intervals.
\end{defn}
\begin{thm} \label{thm:proper}  {\em \cite{LO} }A graph $G$ is a proper interval graph if and only if $G$
has a proper interval ordering.
\end{thm}
Proper interval graphs are strictly related to  \emph{closed graphs}, introduced  by Herzog et al. in \cite{HH}. 
Recently, in \cite{CR2}, the authors have shown that closed graphs and proper interval graphs are synonyms via a graph isomorphism involving some technical results contained in \cite{HEH}.  Here, we show that the above connection between these classes of graphs can be obtained by vertex orderings.

\begin{defn} Let $G$ be a graph on the vertex set $[n]$, $G$ is \emph{closed with respect to the given labeling}, if the following condition
is satisfied:\\
for all $\{i,j\}, \{k, \ell\}\in E(G)$ with $i<j$ and $k<\ell$, one has $\{j,\ell\}\in E(G)$ if $i=k$ but $j \neq \ell$, and $\{i,k\}\in E(G)$ if $j=\ell$, but $i \neq k$.

$G$ is a \emph{closed} graph if there exists a labeling for which is closed.
\end{defn}

\begin{defn} \label{def:closed} Let $\sigma$ be a vertex ordering of a graph $G$ and let
\[v_1 <_{\sigma}v_2 <_{\sigma} \cdots <_{\sigma} v_n,\]
be the ordering of its vertices. Ordering $\sigma$ is called \emph{closed} if for all edges $\{v_i, v_j\}$ and $\{v_k, v_{\ell}\}$ with $i<j$ and $k< \ell$, one has $\{v_j, v_{\ell}\} \in E(G)$ if $i=k$, but $j \neq \ell$ and $\{v_i, v_k\} \in E(G)$ if $j=\ell$, but $i \neq k$.
\end{defn}
Hence, one immediately obtains the following characterization of a closed graph by a closed ordering (see also \cite{CE}).
\begin{prop} \label{prop:closed} A graph $G$ is a closed graph if and only if $G$
has a closed ordering.
\end{prop}

The above proposition yields the following corollary.
\begin{cor} \label{cor:equiv} Let $G$ be a graph.  The following statements are equivalent:
\begin{enumerate}
\item[1.] $G$ is a closed graph;
\item[2.] $G$ has a closed ordering;
\item[3.] $G$ has a proper interval ordering;
\item[4.] $G$ is a proper interval graph.
\end{enumerate}
\end{cor}
\begin{proof} 1. $\Longleftrightarrow$ 2.  follows from Theorem \ref{thm:proper};\\ 
2. $\Longleftrightarrow$  3. can be found in \cite[Proposition 1.8]{KM};\\ 
3. $\Longleftrightarrow$ 4. follows from Proposition \ref{prop:closed}. 
\end{proof}
\begin{rem} In \cite{CR2}, the isomorphism between the class of closed graphs and the class of proper interval graphs 
has been proved by using the relevant characterization of Herzog et al. \cite[Theorem 2.2]{HEH} of a closed graph by its clique complex. In this paper, we recover the isomorphism only by vertex orderings. Moreover, one can see that the \emph{characterization} of Herzog et al. is contained in the constructive proof of the Roberts characterization of a proper interval graph in \cite[Theorem 1]{FG}. 
\end{rem}

\section{Closed graphs  via proper interval graphs} \label{sec:inter}
In this Section, we discuss a recent result on closed graphs \cite{CE} via some properties of proper interval graphs.

In order to accomplish this task we need to recall some notions from the graph theory. 

Let $G$ be a  graph. Denote by $\{u, v\}$ an indirected edge of $G$ and by $(u, v)$ a directed edge (arrow) of $G$.
A graph $G$ is called \emph{mixed} if $G$ has some directed edges (arrows) and some undirected edges such that if G
contains the directed edge $(x,y)$, then it contains neither the directed edge $(y,x)$ nor the
undirected edge $\{x,y\}$. The \emph{inset}, respectively the \emph{outset}, of a vertex $v$ in a mixed graph
$G$ is the set of all vertices $u \in V(G)$ for which $(u,v)$, respectively $(v,u)$, is a directed edge
of $G$. Note that the class of mixed graphs without directed edges
is precisely the class of undirected graphs and the class of mixed graphs without
undirected edges is precisely the class of oriented graphs.
Let $D$ be a mixed graph. If $(x,y)$ is a directed edge of $D$, then we say that $x$ dominates $y$ and write $x\rightarrow y$.
A graph $D$ is an \emph{orientation} of an undirected graph $G$ if $D$ is obtained from $G$ by orienting
each edge $\{u, v\} \in E(G)$ as an arrow  $(u, v)$ or $(v, u)$. A directed graph is called an \emph{oriented} graph if it is the orientation of an undirected graph.
A \emph{straight enumeration} of an oriented graph $D$ is a linear ordering $\{v_1, v_2,\ldots ,v_n\}$ of its vertices such that for each $i$ there exist nonnegative integers $h$ and $k$ such that
the vertex $v$ has inset $\{v_{i-1}, v_{i-2},\ldots, v_{i-h}\}$ and outset $\{v_{i+1}, v_{i+2},\ldots, v_{i+k}\}$. 
An oriented graph which admits a straight enumeration is called \emph{straight}.
An undirected graph is said to have a \emph{straight orientation} if it admits an orientation
which is a straight oriented graph.

Let $G$ be an undirected graph. For any vertex $v$, let $N(v)$ be the neighborhood
of $v$, \emph{i.e.}, the set of vertices which are adjacent to $v$. The closed neighborhood of $v$
is the set $N[v] = N(v)\cup \{v\}$. We define an equivalence relation on $V(G)$ in which
$a$ and $b$ are equivalent just if $N[a]= N[b]$.  
If the vertices $a$ and $b$ of an edge $\{a,b\}$ are equivalent, we call the edge $\{a,b\}$ \emph{balanced};
otherwise, $\{a,b\}$ is an unbalanced edge. We say that $G$ is \emph{reduced} if there are no balanced
edges, \emph{i.e.}, if distinct vertices have distinct closed neighborhoods.
The underlying graph of a mixed graph $D$ is the undirected graph $G(D)$ with the
vertex set $V(D)$ in which $\{x,y\}$ is an edge of $G(D)$ only if it is a directed or undirected
edge of $D$. We say that $D$ is connected if $G(D)$ is connected. We say that $D$ is reduced
if $G(D)$ is reduced. 
A \emph{straight mixed graph} $H$ is a mixed graph obtained from a reduced straight oriented graph $R$ by the  substitution operation which replaces each vertex $v$ of $R$ by a complete graph $T_v$, with
each vertex of $T_v$ dominating each vertex of $T_u$ if and only if $v\rightarrow u$ in $R$. Finally, the
\emph{full reversal} of an oriented graph $D$ is the operation of reversing the directions of all
oriented edges (arrow) of $D$. More details on this subject can be found in \cite{DHH}.

We quote the next results from \cite{DHH}.
\begin{prop} \label{thm:prop1} {\em\cite[Corollary 2.2]{DHH}} A graph is a proper interval graph if and only if it has a straight orientation.
\end{prop}
\begin{prop} \label{thm:prop2} {\em\cite[Corollary 2.5, Proposition 4.2]{DHH}} Let $G$ be a connected proper interval graph.
\begin{enumerate}
\item[1.] $G$ is uniquely orientable as a straight mixed graph up to full reversal.
\item[2.] If $H$ is a straight mixed orientation of a connected subgraph of $G$, and $v\in V(G)$ but $v\notin V(H)$. Then the subgraph of $G$ induced by $v$ and the vertices of $H$ is a proper interval graph.

\end{enumerate}
\end{prop}
In \cite{CE}, the authors have proved a new characterization of the class of closed graphs by a property that they have called \emph{narrow}. Given vertices $v,w$ of $G$ satisfying $d(v,w) = \diam(G)$, a shortest path
connecting $v$ and $w$ is called a \emph{longest shortest path} of $G$.
\begin{defn} A graph $G$ is \emph{narrow} if for every $v \in V (G)$ and every
longest shortest path $P$ of $G$, either $v \in V (P)$ or there is $w \in V (P)$ with $\{v,w\}\in
E(G)$.
\end{defn}

In other words, a connected graph is narrow if every vertex is distance at most
one from every longest shortest path. 
\begin{thm} \label{thm:narrow} {\em \cite[Corollary 1.5]{CE}} Let $G$ be a graph. $G$ is closed if and only if it is chordal, claw-free and narrow.
\end{thm}
\begin{prop} \label{pro:narrow}  A narrow graph $G$ is both net-free and tent-free.
\end{prop}
\begin{proof} In fact, if $G$ contains a copy of the net
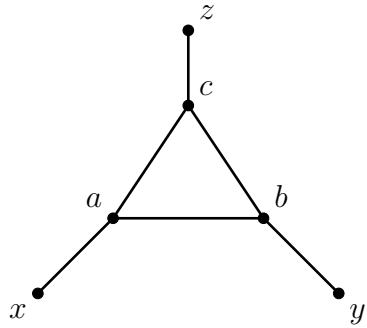
\begin{figure}[H]
\centering
\begin{tikzpicture}
\centering
\filldraw (-0.5,-0.5) circle (2pt) node[above left][black] {$a$}
(-1.5,-1.5) circle (2pt) node[below left][black] {$x$}
(2.5,-1.5) circle (2pt) node[below right][black] {$y$}
(0.5,2) circle (2pt) node[above right][black] {$z$}
(0.5,1) circle (2pt) node[above right][black] {$c$}
(1.5,-0.5) circle (2pt) node[above right][black] {$b$};
\draw [line width=1pt] (-0.5,-0.5) -- (0.5,1.0) node[midway,left] {};
\draw [line width=1pt] (0.5,1.0) -- (1.5,-0.5) node[midway,right] {};
\draw [line width=1pt] (-0.5,-0.5) -- (1.5,-0.5) node[midway,right] {};
\draw [line width=1pt] (-1.5,-1.5) -- (-0.5,-0.5) node[midway,left] {};
\draw [line width=1pt] (1.5,-0.5) -- (2.5,-1.5) node[midway,right] {};
\draw [line width=1pt] (0.5,1.0) -- (0.5,2) node[midway,right] {};
\end{tikzpicture}
\end{figure}
\par\noindent
as an induced subgraph, then the narrowness fails since the vertex $x$ is distance $2$ from the longest shortest path $z,c,b,y$. If $G$ contains a copy of the tent
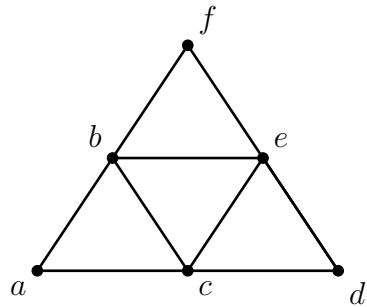
\begin{figure}[H]
\centering
\begin{tikzpicture}
\centering
\filldraw (0,0) circle (2pt) node[below left][black] {$a$}
(1,1.50) circle (2pt) node[above left][black] {$b$}
(2,0) circle (2pt) node[below right][black] {$c$}
(3,1.5) circle (2pt) node[above right][black] {$e$}
(2,3)circle (2pt) node[above right][black] {$f$}
(4,0)circle (2pt) node[below right][black] {$d$};
\draw [line width=1pt] (0,0) -- (1,1.50) node[midway,left] {};
\draw [line width=1pt] (1,1.50) -- (2,0) node[midway,right] {};
\draw [line width=1pt] (0,0) -- (3,0) node[midway,right] {};
\draw [line width=1pt] (2,0) -- (3,1.50) node[midway,right] {};
\draw [line width=1pt] (1,1.50) -- (3,1.5) node[midway,right] {};
\draw [line width=1pt] (3,0) -- (4,0) node[midway,right] {};
\draw [line width=1pt] (3,1.50) -- (4,0) node[midway,right] {};
\draw [line width=1pt] (1,1.5) -- (2,3) node[midway,right] {};
\draw [line width=1pt] (2,3) -- (4,0) node[midway,right] {};

\end{tikzpicture}
\end{figure}
\par\noindent as an induced subgraph, then the narrowness fails since the vertex $a$ is distance $2$ from the longest shortest path $d,e,f$.
 
\end{proof}
The  next result underlines once again the isomorphism between closed graphs and proper interval graphs.
\begin{thm} \label{thm:narrow2} Let $G$ be a chordal claw-free graph. Then $G$ is narrow if and only if it  is net-free or tent-free.
\end{thm}
\begin{proof} The necessary condition follows from Proposition \ref{pro:narrow}. Conversely, suppose there exists a graph $G$ which is chordal, claw-free, net-free, tent-free and not narrow. Assume also that $G$ has the minimal numbers of vertices. Therefore, $G$ is connected. Let $v$ be a vertex of $G$ such that $G-\{v\}$ remains connected. Since $G-\{v\}$ is chordal, claw-free, net-free, or tent-free, the minimality of $\vert V(G)\vert$ implies that $G-\{v\}$ is narrow and so $G-\{v\}$ is a closed graph  (Theorem \ref{thm:narrow}) or, equivalently, a proper interval graph.  Hence, from Proposition \ref{thm:prop2}, $G-\{v\}$ can be oriented as a straight mixed graph $H$ and $G$ itself is orientable as a straight mixed graph. Therefore, from Proposition \ref{thm:prop1}, $G$ is  a proper interval graph (\emph{i.e.}, closed), and consequently a narrow graph. A contradiction.
\end{proof}
\begin{rem} One can observe that combining Theorem \ref{thm:narrow} with Theorem \ref{thm:narrow2}, one gets that \emph{a graph $G$ is closed if and only if it is chordal, claw-free, net-free, or tent-free}, and this is one of the \emph{classical} characterizations of a proper interval graph \cite{MCG}.
\end{rem}



\begin{thebibliography}{99}
\bibitem{BL} {K.S. Booth,  G.S. Lucker}, \textit{Testing for the consecutive ones property, interval graphs, and graph planarity using PQ-tree algorithms}, \newblock J. Comput. System Sci., {\bf 13}(1976), 335--379.
\bibitem{CE} {D.A. Cox, A. Erskine}, \textit{On closed graphs}, Ars Combinatoria, to appear, preprint available at  arXiv:1306.5149v1 [math.CO].
\bibitem{CR}{M. Crupi, G. Rinaldo}, \textit{Binomial edge ideals with quadratic Gr\"obner bases}, \newblock The Elec. J. of Comb. {\bf 8}, P\#211  (2011), 1--13.
\bibitem{CR2}{M. Crupi, G. Rinaldo},  \textit{Closed graphs are proper interval graphs}, \newblock An.St Univ. Ovidius Constantia, {\bf 22}(3) (2014), 37--44.

\bibitem{DHH} {X. Deng, P. Hell, J. Huang,} \textit{Linear-time representation algorithms for proper circular-arc graphs and proper interval graphs}, SIAM J. Comput., {\bf 25}(2) (1996), 390--403.
\bibitem{FG} {F. Gardi,}  \textit{ The Roberts characterization of
proper and unit interval graphs}, \newblock Discr. Math., {\bf 307}(22) (2007), 2906--2908.

\bibitem{MCG} {M.C. Golumbic,} \textit{ Algorithm Graph Theory and Perfect graphs. Academic Press}, New York, N.Y., 1980.
\bibitem{GH} {P.C. Gilmore, A.J. Hoffman,}  \textit{ A characterization of comparability graphs and of interval graphs}, Canad. J. Math., {\bf16} (1964), 539--548.

\bibitem{HEH}{V. Ene, J. Herzog, T. Hibi,} \textit{Cohen-Macaulay binomial edge ideals},
\newblock Nagoya Math. J. ,{\bf 204} (2011), 57--68.

\bibitem{Ha} {M. Habib, R. McConnel C. Paul, L. Viennot,}  \textit{Lex-BFS and partition refinement, with applications to transitive orientation, interval graph recognition, and consecutive one testing},
\newblock Theoret. Comput. Sci., {\bf 234} (2000), 59--84.
\bibitem{HMP} {P. Heggernes, D. Meister, C. Papadopoulos,} \textit{A new representation of proper interval graphs with an application to clique-width}, DIMAP Workshop on Algorithmic Graph Theory 2009, Electron. Notes Discrete Math., {\bf 32}, (2009), 27--34.
\bibitem{HH}{J. Herzog, T. Hibi, F. Hreinsdottir, T. Kahle, J. Rauh,}  \textit{Binomial edge ideals and conditional independence statements},
\newblock Adv. in Appl. Math., {\bf 45} (2010), 317--333.
\bibitem{LO} {P. J. Looges, S. Olariu, } \textit{ Optimal greedy algorithms for indifference graphs},
\newblock Comput. Math. Appl., {\bf 25} (1993), 15--25.

\bibitem{KM} {K. Matsuda, } \textit{ Weakly closed graphs and F-purity of binomial edge ideals},
\newblock preprint available at arXiv:1209.4300v1 [Math:AC].
\bibitem{SO} {S. Olariu,} \textit{ An optimal greedy euristic to color interval grphs graphs},
\newblock Inform. Process. Lett., {\bf 31} (1991), 21--25.

\bibitem{PD}{B.S. Panda, S.K.  Das,}  \textit{A linear time recognition algorithm for proper
interval graphs}, Inform. Process. Lett., {\bf 87} (2003), 153--161.
\bibitem{Roberts1} {F.S. Roberts,} \textit{Graph theory and its application to problem of society}, SIAM Press, Philadelphia (PA), 1978.

\end{thebibliography}
\end{document}